\documentclass[12pt,a4paper]{amsart}
\usepackage{amssymb,amsmath,amsthm, amscd, enumerate, mathrsfs, calligra}
\usepackage{enumitem}
\usepackage{mathrsfs}
\usepackage{mathtools}
\usepackage[all,cmtip]{xy}
\usepackage[T1]{fontenc}
\usepackage{color}
\usepackage{physics}
\usepackage{stmaryrd}

\usepackage[
	colorlinks,
	linkcolor=blue,
	anchorcolor=blue,
	citecolor=green
]{hyperref}
\usepackage[capitalise]{cleveref}

\usepackage[margin=1in]{geometry}

\newtheorem{thm}{Theorem}[section]
\newtheorem{lem}[thm]{Lemma}

\newtheorem{prop}[thm]{Proposition}

\theoremstyle{definition}
\newtheorem{defn}[thm]{Definition}
\newtheorem{rmk}[thm]{Remark}

\newtheorem*{ack}{Acknowledgments}
\newtheorem*{data}{Data availability statement}
\newtheorem*{conflict}{Conflict of interest statement}

\newtheorem*{notation}{Notation and Conventions}

\DeclareMathOperator{\Aut}{Aut}

\DeclareMathOperator{\grp}{grp}

\DeclareMathOperator{\id}{id}

\DeclareMathOperator{\Ker}{Ker}

\DeclareMathOperator{\ord}{ord}

\DeclareMathOperator{\pr}{pr}

\DeclareMathOperator{\rord}{rord}

\DeclareMathOperator{\Sym}{Sym}

\DeclareMathOperator{\sHom}{\mathscr{H}\kern -.3pt \mathit{om}}

\newcommand{\vep}{\varepsilon}
\newcommand{\vph}{\varphi}

\newcommand{\bC}{\mathbb{C}}

\newcommand{\bP}{\mathbb{P}}
\newcommand{\bQ}{\mathbb{Q}}

\newcommand{\bZ}{\mathbb{Z}}

\newcommand{\cO}{\mathcal{O}}

\newcommand{\fS}{\mathfrak{S}}

  {\end{list}}

\begin{document}

\title[A biregular classification of smooth quotients of abelian surfaces]
{A biregular classification of smooth quotients of abelian surfaces}

\author{Takahiro Shibata}
\address{Fukushima University, 1, Kanayagawa, Fukushima-city, 960-1296, Japan}
\email{t-shibata@educ.fukushima-u.ac.jp}

\maketitle

\begin{abstract}
We classify smooth projective surfaces that are quotients of abelian surfaces by finite groups.
We also record the orbifold boundary associated to a quotient map and describe its structure explicitly in the case of augmented irregularity one.
\end{abstract}

\tableofcontents

\section{Introduction}\label{sec_intro}

Throughout this paper, we work over the complex number field $\bC$.
In this paper, we give a biregular classification of smooth projective surfaces 
that are finite quotients of abelian surfaces, 
that is, quotients of abelian surfaces by finite groups.

Smooth finite quotients of abelian varieties are investigated in several papers.
Auffarth, Lucchini Arteche, and Quezada \cite{ALA20}, \cite{ALAQ22} studied 
when smooth finite quotients of abelian varieties are isomorphic to the projective spaces.
Lange \cite{Lan01} and Catanese--Demleitner \cite{CD20} completed the classification of hyperelliptic threefolds.
Here a hyperelliptic variety is a smooth projective variety that is not an abelian variety but a finite \'etale quotient of an abelian variety.

Recently, Auffarth and Lucchini Arteche \cite[Theorem 1.3]{ALA22} proved that 
any smooth finite quotient of a complex torus is a fibration of a product of projective spaces over a complex torus or a hyperelliptic manifold (i.e.~a complex manifold that is not a complex torus but a finite \'etale quotient of a complex torus).
Our classification gives a precise picture of their result in dimension two and in the algebraic setting.

In arbitrary dimension, Catanese \cite{Cat24}
characterizes finite quotients of abelian varieties and complex tori
in terms of orbifold classifying spaces.

Yoshihara \cite{Yos95} gave a birational classification of finite quotients of abelian surfaces.
Our classification can be seen as a biregular classification of finite quotients of abelian surfaces,
restricted to smooth ones.

Our main result in this paper is the following: 

\begin{thm}\label{thm_main}
Let $X$ be a smooth projective surface.
Then $X$ is the quotient of an abelian surface by a finite group if and only if 
$X$ is isomorphic to one of the following:
\begin{itemize}
\item $\bP^2$.
\item $\bP^1 \times \bP^1$.
\item The projective bundle $\bP_E(\cO_E \oplus L)$ where $E$ is an elliptic curve and $L$ is a torsion line bundle of order at most $4$ on $E$.
\item The symmetric square $\Sym^2 E$ of an elliptic curve $E$.
\item Abelian surfaces.
\item Hyperelliptic surfaces.
\end{itemize}
\end{thm}

For the definition of augmented irregularity $q^\circ(X)$, 
see Definition \ref{defn_qcirc}.

It is useful to keep track not only of the surface $X$ but also of the quotient map.
Let $\pi:A\to X=A/G$ be a quotient map from an abelian surface and let
$B_1,\dots,B_r$ be the irreducible components of its branch divisor.
Since $\pi$ is a quotient by a finite group, the induced extension of function fields
$\bC(A)/\bC(X)$ is Galois. Hence the ramification indices of the prime divisors of $A$
lying over a fixed $B_j$ are all equal.
We denote this common ramification index by $m_j$ and set
\[
\Delta_\pi:=\sum_{j=1}^r\left(1-\frac{1}{m_j}\right)B_j.
\]
We call $\Delta_\pi$ the \textit{orbifold boundary} of the quotient map $\pi$.
For the definition of orbifolds, see \cite{Cat24}.
The Hurwitz formula gives
\[
K_A=\pi^*(K_X+\Delta_\pi),
\qquad
K_X+\Delta_\pi\sim_{\bQ}0.
\]

The following theorem gives a more precise description of the quotient map in the case
$q^\circ(X)=1$.

\begin{thm}\label{thm_orbstr}
Let $X$ be a smooth projective surface that is a quotient of an abelian surface by a finite group,
and assume that $q^\circ(X)=1$.
Then there exist elliptic curves $E$ and $F$, a finite subgroup
$G<\Aut(E\times F)$, and a non-trivial normal cyclic subgroup $N\lhd G$
such that $X\cong(E\times F)/G$ and,
Writing
\[
\rho:E\times F\to Y:=(E\times F)/N,
\quad
\eta:Y\to X
\]
for the induced quotient maps, the following hold:
\begin{enumerate}[label=\textup{(\arabic*)}]
\item
$N$ acts trivially on $E$, and $Y\cong E\times\bP^1$.
\item
The morphism $\eta$ is finite \'etale.
\item
The order of $N$ is $2$, $3$, $4$, or $6$.
If the order of $N$ is $2$, $3$, $4$, or $6$, then the ramification indices of
\[
F\to F/N\cong\bP^1
\]
over its branch points are, respectively,
\[
(2,2,2,2),\quad
(3,3,3),\quad
(2,4,4),\quad
(2,3,6).
\]
\end{enumerate}
Moreover, let $D_1,\dots,D_s$ be the branch divisors of $\rho$, and let $e_i$
be the ramification index along $D_i$.
Under the identification $Y\cong E\times\bP^1$, each $D_i$ is of the form
$E\times\{p_i\}$, and
\[
\eta^*\Delta_\pi
=
\sum_{i=1}^s\left(1-\frac1{e_i}\right)D_i,
\]
where $\pi=\eta\circ\rho:E\times F\to X$.
\end{thm}

In Section \ref{sec_prelim}, we give notation and lemmas needed.
In Section \ref{sec_proof}, we prove Theorems \ref{thm_main} and \ref{thm_orbstr}.
The proof of Theorem \ref{thm_main} is divided by augmented irregularity.

\begin{ack}
I thank Yohsuke Matsuzawa and Masataka Iwai for discussions and comments,
and the referee for valuable comments and suggestions.
\end{ack}

\section{Preliminaries}\label{sec_prelim}

\begin{notation}
\begin{itemize}
$ \, $

\item For a normal projective variety $X$,
$q(X)=h^1(X, \mathcal O_X)$ denotes the \textit{irregularity} of $X$.

\item For a variety $X$, $\Aut(X)$ denotes the automorphism group of $X$.
When $X$ is an algebraic group, $\Aut_{\grp}(X)$ denotes the group of the automorphisms of algebraic group.

\item When considering an abelian variety or an elliptic curve, we assume that its group structure is fixed.

\item For an abelian variety $A$ and $n \in \bZ$,
we define a morphism $[n]: A \to A$ by $[n](x)=nx$.

\item For an abelian variety $A$ and $n \in \bZ_{\ge 1}$,
$A[n]$ denotes the set of $n$-torsion points on $A$.

\item For an abelian variety $A$ and a point $a \in A$,
$t_a: A \to A$ denotes the translation map by $a$.

\item Let $A$ be an abelian variety and $g: A \to A$ a morphism.
Then we can write $g=t_a \circ \phi$ with $\phi: A \to A$ being a group homomorphism.
We call $\phi$ the \textit{holonomy part} of $g$.

\item Given a group $G < \Aut(A)$, the set $G_0 \subset \Aut_{\grp}(A)$ of the holonomy parts of the elements of $G$ forms a group.
We call $G_0$ the \textit{holonomy part} of $G$.
\end{itemize}

\end{notation}

\begin{defn}[cf.~{\cite[p.~27]{DPS94}}, {\cite[Definition 2.6]{NZ10}}]\label{defn_qcirc}
Let $X$ be a normal projective variety.
We define the \textit{augmented irregularity} of $X$: 
$$q^\circ(X)=\max_Y q(Y),$$
where $Y$ is taken to be all normal projective varieties which are finite quasi\'etale covers of $X$.
\end{defn}

We use the following structure theorem on finite quotients of abelian varieties.

\begin{thm}[cf.~{\cite[Theorem 1.1]{DHP08}}, {\cite[Theorem 1.2 and Proposition 4.5]{Shi25}}]\label{thm_decomp}
Let $X$ be a finite quotient of an abelian variety.
Then there exists a finite quasi\'etale cover $\theta: B \times F \to X$ such that  
$B$ is an abelian variety and $F$ is a finite quotient of an abelian variety that is $\bQ$-Fano.
In this situation, we have $q^\circ(X)=\dim B$ and $\kappa(-K_X)=\dim F$.
\end{thm}

From here to the end of this section, we prepare some lemmas needed for the proof of Theorem \ref{thm_main}.

\begin{lem}\label{lem_persubvar}
Let $A$ be an abelian variety and $G$ a finite subgroup of $\Aut A$.
Assume that there exists an abelian subvariety $B \subset A$ 
such that $g(B)$ is a translate of $B$ for any $g \in G$.
Then there exists another abelian subvariety $C \subset A$ such that 
the addition map $\mu: B \times C \to A$ is an isogeny and 
there exists a finite subgroup $\tilde{G}$ of $\Aut(B \times C)$ such that 
$A/G \cong (B \times C)/\tilde{G}$.
\end{lem}

\begin{proof}
Let $G_0$ be the holonomy part of $G$, then $B$ is $G_0$-invariant.
There exists a $G_0$-invariant abelian subvariety $C \subset A$ 
such that the addition map $\mu: B \times C \to A$ is an isogeny 
(cf.~\cite[Lemma 2.14]{Shi25}).
Set 
$$\tilde{G}=\{(t_b \circ \vph|_B) \times (t_c \circ \vph|_C) \mid b \in B,\ c \in C,\ \vph \in G_0,\ t_{b+c} \circ \vph \in G \} \subset \Aut(B \times C).$$

We have the following diagram:
$$
\xymatrix{
B \times C \ar[r]^{p} \ar[d]_{\mu} & (B \times C)/\tilde{G} \ar[d]^{\overline{\mu}}  \\
A \ar[r]^q & A/G
}
$$
Take any two points $p(y,z), p(y',z') \in (B \times C)/\tilde{G}$ with 
$\overline{\mu}(p(y,z))=\overline{\mu}(p(y',z'))$.
Then $q(y+z)=q(y'+z')$, so 
there exists $t_a \circ \vph \in G$ ($\vph \in G_0$) such that 
$$y'+z'=(t_a \circ \vph)(y+z)=\vph(y)+\vph(z)+a.$$
Take $b \in B,\ c \in C$ such that $a=b+c$.
Then $y'+z'=\vph(y)+\vph(z)+b+c$.
Set 
$$d=\vph(y)-y'+b \ (=z'-\vph(z)-c).$$
Then $d \in B \cap C$, so 
$\tilde{g}=(t_{b-d}\circ \vph|_B) \times (t_{c+d} \circ \vph|_C) \in \tilde{G}$.
We have 
\begin{align*}
\tilde{g}(y,z)
&=(\vph(y)+b-d, \vph(z)+c+d) 
=(y', z').
\end{align*}
So $p(y,z)=p(y',z')$.
Therefore $\overline{\mu}$ is an isomorphism.
\end{proof}

\begin{lem}\label{lem_min}
Let $X$ be a smooth finite quotient of an abelian surface.
Then $X$ is a minimal surface i.e.~$X$ has no $(-1)$-curve.
\end{lem}

\begin{proof}
Let $\pi: A \to X$ be the quotient map from an abelian surface $A$.
If $X$ has an effective divisor $D$ with negative self-intersection, then 
so is $\pi^*D$.
But any effective divisor on $A$ is nef, so it is a contradiction.
Hence $X$ has no $(-1)$-curve.
\end{proof}

\begin{lem}\label{lem_cancel}
Let $E$ be an elliptic curve and $\vph, \psi$ group automorphisms of the same order on $E$.
Then $\psi=\vph$ or $\psi=\vph^{-1}$.
\end{lem}

\begin{proof}
It is well-known that $\Aut_{\grp}(E)$ is isomorphic to 
$\bZ/2$, $\bZ/4$ or $\bZ/6$.
Then the assertion is easy to verify.
\end{proof}

\begin{lem}\label{lem_power}
Let $E$ be an elliptic curve and $\vph$ a group automorphism on $E$ with $\ord(\vph)=r \ge 2$.
Then $1+\vph+\vph^2+\cdots+\vph^{r-1}=0$.
\end{lem}

\begin{proof}
We have 
$0=1-\vph^r=(1-\vph)(1+\vph+\vph^2+\cdots+\vph^{r-1})$. 
Now $\vph \neq 1$, so the above equality implies $1+\vph+\vph^2+\cdots+\vph^{r-1}=0$.
\end{proof}

\begin{lem}\label{lem_table}
Let $E$ be an elliptic curve and $\vph \neq 1$ a group automorphism on $E$.
Then we have the following table:
\begin{center}
\begin{tabular}{|l||c|c|c|c|} \hline
The order of $\vph$ & $2$ & $3$ & $4$ & $6$ \\ \hline 
The number of fixed points of $\vph$ & $4$ & $3$ & $2$ & $1$ \\ \hline
\end{tabular}
\end{center}
\end{lem}

\begin{proof}
The number of fixed points of $\vph$ is equal to $\deg(\vph-1)$,
which is computed as $|\alpha-1|^2$, where $\alpha \in \bC$ is the analytic representation of $\vph$.
Now $\alpha=-1, \omega, i, -\omega$ when the order of $\vph$ is $2, 3, 4, 6$, respectively.
So the assertion follows from easy computations.
\end{proof}

\begin{lem}\label{lem_split}
Let $E$ be an elliptic curve.
Then any automorphism on $E \times \bP^1$ is of the form $\vph \times \psi$ ($\vph \in \Aut(E),\ \psi \in \Aut(\bP^1)$).
\end{lem}

\begin{proof}
Take any $g \in \Aut(E \times \bP^1)$.
Then $g$ induces $\vph \in \Aut(E)$ such that $\pr_1 \circ g=\vph \circ \pr_1$.
For $x \in E$, 
the morphism $g_x: \bP^1 \cong \{x\} \times \bP^1 \overset{g}{\to} \{\vph(x)\} \times \bP^1 \cong \bP^1$ is an automorphism on $\bP^1$.
This correspondence gives a  morphism $\Phi: E \to \Aut(\bP^1)$. 
But $E$ is complete and $\Aut(\bP^1)$ is affine, so $\Phi$ is constant.
\end{proof}

We give a proof of the following well-known result for the reader's convenience:
 
\begin{lem}[cf.~{\cite[1.1]{Dol09}}]\label{lem_unique}
Let $G$ be a finite group that is either a cyclic group or a dihedral group.
Then the faithful action of $G$ on $\bP^1$ is unique up to conjugacy.
\end{lem}

\begin{proof}
We regard $G$ as a subgroup of $\Aut(\bP^1)$.
First, assume that $G$ is the $n$-th cyclic group generated by $g \in G$.
We may assume that $n \ge 2$.
Then $g$ has exactly two fixed points.
Replacing $G$ by a conjugate, it follows that $0, \infty$ are fixed points of $g$.
So $g$ is the multiplication (in the affine coordinate) by a primitive $n$-th root of unity $\zeta_n \in \bC$.
Therefore the faithful action of the $n$-th cyclic group to $\bP^1$ is, up to conjugacy, the multiplication by $n$-th roots of unity.

Second, assume that $G$ is the $n$-th dihedral group generated by $g, h \in G$ where $\ord(g)=n$ and $\ord(h)=2$.
We may assume that $n \ge 2$.
As above, after conjugacy, $g$ is the multiplication by a primitive $n$-th root of unity $\zeta_n \in \bC$.
Let $a, b \in \bP^1$ be the two fixed points of $h$.
Now we have $ghg=h$, so 
$\{h(0), h(\infty) \}=\{0, \infty\}$.
If $h(0)=0$ and $h(\infty)=\infty$, then $h$ is the multiplication by $-1$,
which implies that $G$ is generated by the multiplication by $-\zeta_n$, a contradiction.
So $h(0)=\infty$ and $h(\infty)=0$,
which implies that $h(z)=\frac{a}{z}$ (in the affine coordinate) for some $a \in \bC \setminus \{0\}$.
Replacing the affine coordinate $z$ by $w=\sqrt{a^{-1}}z$, we have $g(w)=\zeta_n w$ and $h(w)=\frac{1}{w}$.
So the assertion follows.
\end{proof}

\section{Proof}\label{sec_proof}

\subsection{The $q^\circ(X)=0$ or $2$ case}

\begin{prop}\label{prop_qcirc=0}
Let $X$ be a smooth finite quotient of an abelian surface with $q^\circ(X)=0$.
Then $X$ is isomorphic to $\bP^2$ or $\bP^1 \times \bP^1$.
\end{prop}

\begin{proof}
Theorem \ref{thm_decomp} implies that $X$ is a del Pezzo surface.
Moreover $X$ is a minimal surface (cf.~Lemma \ref{lem_min}), so the assertion follows.
\end{proof}

\begin{prop}\label{prop_qcirc=2}
Let $X$ be a smooth finite quotient of an abelian surface with $q^\circ(X)=2$.
Then $X$ is isomorphic to an abelian surface or a hyperelliptic surface.
\end{prop}

\begin{proof}
Theorem \ref{thm_decomp} implies that $X$ admits a finite \'etale cover $\theta: A \to X$ from an abelian surface $A$.
This implies that $\chi(\cO_X)=(\deg \theta)^{-1}\chi(\cO_A)=0$.
Now $X$ is a minimal surface with $\kappa(X)=0$ and $\chi(\cO_X)=0$, so the assertion follows.
\end{proof}

\begin{rmk}\label{rmk_qcirc02}
All of the surfaces $\bP^2$, $\bP^1 \times \bP^1$, abelian surfaces, and hyperelliptic surfaces are finite quotients of abelian surfaces.
Let $E$ be an elliptic curve.
Then $\bP^2$ is the quotient of 
$$\{ (x,y,z) \in E^3 \mid x+y+z=0 \} \ (\cong E^2)$$
by the natural action of the third symmetric group $\fS_3$ (cf.~\cite[2.3]{Auf17}),
and $\bP^1 \times \bP^1$ is the quotient of $E \times E$ by a group of automorphisms generated by $1 \times [-1]$ and $[-1] \times 1$.
Abelian surfaces are trivial quotients of themselves.
It is well-known that hyperelliptic surfaces are finite quotients of abelian surfaces (see e.g.~\cite{Bea96}).
\end{rmk}

\subsection{The $q^\circ(X)=1$ case}

From now on, we fix a smooth finite quotient $X$ of an abelian surface with $q^\circ(X)=1$.

\begin{lem}\label{lem_P1-bdl}
$X$ is a $\bP^1$-bundle over an elliptic curve.
\end{lem}

\begin{proof}
By Theorem \ref{thm_decomp}, there exists a finite \'etale cover $\theta: E \times \bP^1 \to X$ where $E$ is an elliptic curve.
Take the Galois closure $Z \overset{\vep}{\to} E \times \bP^1 \overset{\theta}{\to} X$ of $\theta$.
Then $\vep$ is also \'etale and so $Z$ is smooth.
Now $q(Z)=q^\circ(X)=1$, so $Z$ has the Albanese morphism $a: Z \to E_Z$ to an elliptic curve.
The universality of Albanese morphism induces a morphism $\vph: E_Z \to E$ such that 
$\pr_1 \circ \vep=\vph \circ a$.
Taking the pullback of $\pr_1: E \times \bP^1 \to E$ by $\vph$,
we obtain the following commutative diagram:
$$
\xymatrix@C=40pt{
Z \ar@/^20pt/[rr]^{\vep} \ar[r]^{\vep_1}  \ar[dr]_a & E_Z \times \bP^1 \ar@{}[rd]|{\Box} \ar[r]^{\vep_2} \ar[d]_{\pr_1} & E \times \bP^1 \ar[d]^{\pr_1}  \\
 & E_Z \ar[r]^{\vph}  & E
}
$$
Now $\vep_1$ is a finite \'etale cover, 
so its restriction to any fiber is an \'etale cover from $\bP^1$ to $\bP^1$,
which must be an automorphism.
This implies that $\vep_1$ is an isomorphism.
Replacing $E \times P^1$ by $Z=E_Z \times \bP^1$,
we may assume that $\theta: E \times \bP^1 \to X$ is Galois.

Let $G$ be the Galois group of $\theta$.
The action of $G$ on $E \times \bP^1$ descends to an action on $E$.
So we obtain the following commutative diagram:
$$
\xymatrix{
 E \times \bP^1 \ar@{}[rd] \ar[r]^{\theta} \ar[d]_{\pr_1} & X \ar[d]^{\pi}  \\
E \ar[r] & E/G
}
$$
Now $\pi$ is a $\bP^1$-fibration over $E/G$.
But $X$ is a minimal surface, so $\pi$ is a $\bP^1$-bundle over $E/G$.
If $E/G=\bP^1$, then $X$ is a Hirzebruch surface.
This implies that $(K_X^2)=8$ (cf.~\cite[Corollary 2.11]{Har77}), but this contradicts that 
$(K_X^2)=(\deg \theta)^{-1} (K_{E \times \bP^1}^2)=0$.
So $E/G$ is an elliptic curve.
\end{proof}

Set $X=A/G'$ for some abelian surface $A$ and a finite subgroup $G' < \Aut(A)$.
We may assume that $G'$ has no translation.
Let $q: A \to X$ be the quotient map, and $\pi:X \to E_X$ the Albanese morphism.
Changing the group structure of $E_X$ if necessary, we may assume that $\pi \circ q: A \to E_X$ is a group homomorphism.
Let $E \subset A$ be the neutral component of $\Ker(\pi \circ q)$.
Since $\Ker(\pi \circ q)$ is $G'$-invariant, 
$g(E)$ is a translate of $E$ for any $g \in G'$.
Lemma \ref{lem_persubvar} implies that 
$X \cong (E \times F)/G$, where $F$ is another elliptic curve and 
$G < \Aut(E \times F)$ is a finite subgroup consisting of automorphisms of the form $g \times h$ ($g \in \Aut(E),\ h \in \Aut(F)$).

Now Theorem \ref{thm_decomp} implies that $\kappa(-K_X)=1$.
Then, using \cite[Lemma 4.2]{Shi25}, just one of the following holds: 
\begin{itemize}
\item 
There exists $1 \times h \in G$ where $h$ is not a translation, 
and $G$ has no elements of the form $g \times 1$ where $g$ is not a translation.
\item
There exists $g \times 1 \in G$ where $g$ is not a translation, 
and $G$ has no elements of the form $1 \times h$ where $h$ is not a translation.
\end{itemize}
Note that the identity map is regarded as a translation by the origin.
These two conditions are symmetric, so we may assume that the first condition holds.

Take $1 \times h_0 \in G$ where $h_0$ is not a translation.
Suppose that $G$ contains an element $g \times h$ where $g$ is not a translation.
Multiplying  $1 \times h_0$ if necessary, we may assume that $h$ is not a translation.
Now $g \times h$ has a fixed point $z=(x,y) \in E \times F$.
Then the Chevalley--Shephard--Todd theorem implies that 
the action of the stabilizer $G_z$ of $z$ on the tangent space $T_z$ at $z$ is generated by pseudoreflections.
So $G_z$ must contain $g' \times 1$ and $1 \times h'$,
where $g'$, $h'$ have the same holonomy part as $g$, $h$, respectively.
But this contradicts the above condition.
So the first factor of any element of $G$ is a translation.

Set 
\begin{gather*}
N_E=\{t_b \in \Aut(E) \mid b \in E,\ t_b \times 1 \in G\}, \\
N_F=\{t_c \in \Aut(F) \mid c \in F,\ 1 \times t_c \in G\}.
\end{gather*}
Then $N_E \times 1$ is a normal subgroup of $G$, so 
$(E \times F)/G \cong ((E/N_E) \times F)/(G/(N_E \times 1))$.
Therefore we may assume that $N_E=1$.
Similarly we may assume that $N_F=1$.

In summary, now we have: 
\begin{itemize}
\item
$X \cong (E \times F)/G$ where $E, F$ are elliptic curves and $G < \Aut(E \times F)$ is a finite subgroup.
\item 
Any element of $G$ is of the form $t_a \times h$.
\item
$G$ contains no elements of the form $t_a \times 1$ nor $1 \times t_b$ except the identity element.
\end{itemize}
The third condition is equivalent to that any translation $t_a \times t_b$ in $G$ satisfies $\ord(a)=\ord(b)$.
Let $\Delta < G$ be the subgroup of translations.
Clearly it is a normal subgroup.

\begin{defn}\label{defn_rord}
Take $g \in G$.
Write $g=t_a \times (t_b \circ \vph)$ for some $a \in E$,  $b \in F$ and $\vph \in \Aut_{\grp}(F)$.
Let $\rord(g)$ be the order of $\vph$.
\end{defn}

\begin{lem}\label{lem_G/Delta}
$G/\Delta$ is a cyclic group of order $2, 3, 4$ or $6$.
\end{lem}

\begin{proof}
Set $\displaystyle r= \max_{g \in G} \rord(g)$ and fix $g \in G$ with $\rord(g)=r$.
Take any $h \in G$ and suppose $\rord(h)$ does not divide $r$.
This implies that $\Aut_{\grp}(F)=\bZ/6$, $\rord(h)=2$ and $r=3$.
But then $\rord(gh)=6>r$, a contradiction.
So  $\rord(h) | r$.
Set $q=r/\rord(h)$, then $\rord(g^q)=\rord(h)$.
So Lemma \ref{lem_cancel} implies that $h \cdot g^{\pm q} \in \Delta$.
Therefore $G/\Delta$ is the cyclic group generated by the equivalence class of $g$.
\end{proof}

Fix $g \in G$ generating $G/\Delta$, and set $r=\rord(g)$.
Write $g=t_{a_0} \times (t_{b_0} \circ \vph)$.
Note that $\vph \neq 1$ since $G \neq \Delta$ (if $G=\Delta$, then $(E \times F)/G$ is an abelian surface).
So $t_{b_0} \circ \vph$ has a fixed point $p \in F$.
Replacing $G$ by $(1 \times t_p)^{-1}G(1 \times t_p)$,
we may assume that $g=t_{a_0} \times \vph$.
Then $\ord(a_0) \mid r$ by assumption.

\begin{lem}\label{lem_fixed}
For any $\tau=t_a \times t_b \in \Delta$, $b$ is a fixed point of $\vph$.
\end{lem}

\begin{proof}
Lemma \ref{lem_G/Delta} implies that 
$g \circ \tau= \tau' \circ g^m$ for some $\tau'=t_c \times t_d \in \Delta$ and $0 \le m \le r-1$.
Comparing the holonomy parts, we have $m=1$.
So we have 
$$t_{a_0+a} \times (t_{\vph(b)} \circ \vph) =t_{c+a_0} \times (t_d \circ \vph).$$
Therefore $a=c$ and $\vph(b)=d$.
Now the first factor of $\tau$ and $\tau'$ are equal, so $\tau=\tau'$ and $b=d$.
Hence $\vph(b)=b$.
\end{proof}

\begin{lem}\label{lem_comm}
$G=\Delta \times \langle g \rangle$.
\end{lem}

\begin{proof}
For $\tau=t_a \times t_b \in \Delta$,
$$g \circ \tau=t_{a_0+a} \circ (t_{\vph(b)} \circ \vph)
=t_{a_0+a} \circ (t_{b} \circ \vph)=\tau \circ g.$$
Moreover, $g^r=1$ because $g=t_{a_0}\times\vph$ and $\ord(a_0)\mid r$, while $\langle g\rangle\cap\Delta=\{1\}$ by comparison of holonomy parts.
Thus $G=\Delta\times\langle g\rangle$.
\end{proof}

\begin{lem}\label{lem_Delta}
\begin{itemize}
\item[\ ]
\item[(1)]
If $r=2$, then $\Delta \cong \{1\}$, $\bZ/2$ or $(\bZ/2)^2$.
\item[(2)]
If $r=3$, then $\Delta \cong \{1\}$ or $\bZ/3$.
\item[(3)]
If $r=4$, then $\Delta \cong \{1\}$ or $\bZ/2$.
\item[(4)]
If $r=6$, then $\Delta \cong \{1\}$.
\end{itemize}
\end{lem}

\begin{proof}
Assume that $G$ has an element $h =t_p \times (t_q \circ \psi) \in G$ with $\rord(h)=s \ge 2$.
Take any $\tau =t_a \times t_b \in \Delta$. 
Then we have 
$$h^s=t_{sp} \times (t_{q+\psi(q)+\cdots+\psi^{s-1}(q)} \circ \psi^s) \overset{\mathrm{\ref{lem_power}}}{=} t_{sp} \times 1,$$
and
$$(\tau \circ h)^s=t_{s(p+a)} \times (t_{(b+q)+\psi(b+q)+\cdots+\psi^{s-1}(b+q)} \circ \psi^s) \overset{\mathrm{\ref{lem_power}}}{=} t_{s(p+a)} \times 1.$$
So $sp=s(p+a)=0$, and $\ord(a) | s$.
Now $\ord(a)=\ord(b)=\ord(\tau)$, so $\ord(\tau)|s$.
Since $\tau \in \Delta$ is arbitrary, 
$\Delta$ is embedded into $E[s]$ by the projection to the first factor.

Assume that $r=2$.
Then $\Delta$ is embedded into $E[2] \cong (\bZ/2)^2$.
So $\Delta \cong \{1\}$, $\bZ/2$ or $(\bZ/2)^2$.

Assume that $r=3$.
Then $\Delta$ is embedded into $E[3] \cong (\bZ/3)^2$.
So $\Delta \cong \{1\}$, $\bZ/3$ or $(\bZ/3)^2$.
But Lemmas \ref{lem_fixed} and \ref{lem_table} imply that $\# \Delta \le 3$,
so $\Delta \cong \{1\}$ or $\bZ/3$.

Assume that $r=4$.
Applying the preceding argument to $h=g^2$, whose reduced order is $2$, shows that every element of $\Delta$ has order dividing $2$.
Thus $\Delta$ is embedded into $E[2] \cong (\bZ/2)^2$.
So $\Delta \cong \{1\}$, $\bZ/2$ or $(\bZ/2)^2$.
But Lemmas \ref{lem_fixed} and \ref{lem_table} imply that $\# \Delta \le 2$,
so $\Delta \cong \{1\}$ or $\bZ/2$.

Finally, assume that $r=6$.
Applying the preceding argument to $g^3$ and $g^2$, whose reduced orders are $2$ and $3$, respectively, gives $\ord(\tau)\mid2$ and $\ord(\tau)\mid3$ for every $\tau\in\Delta$.
Hence $\Delta=\{1\}$.
\end{proof}

\begin{lem}\label{lem_reduction}
Set $N=\Ker(G \to \Aut(E))$, where $G \to \Aut(E)$ is the projection to the first factor.
\begin{itemize}
\item[(1)]
$N$ is a non-trivial cyclic group.
\item[(2)]
$(E \times F)/N \cong E \times \bP^1$.
\item[(3)]
$G/N \cong \{1\},\ \bZ/2,\ \bZ/3,\ \bZ/4$, or $(\bZ/2)^2$.
\end{itemize}
\end{lem}

\begin{proof}
The action of $G$ is not free: otherwise $E\times F\to X$ would be finite \'etale and hence $q^\circ(X)=2$, contrary to our assumption.
Thus $G \setminus \{1\}$ contains an automorphism with a fixed point.
Since the first factor of every element of $G$ is a translation, such an element must act trivially on the first factor.
Therefore $N \neq \{1\}$.
The claim that $N$ is cyclic is proved as in Lemma \ref{lem_G/Delta}.

Let $1 \times h$ be a generator of $N$.
Then $h$ has non-trivial holonomy part.
So $(E \times F)/N = E \times (F/\langle h \rangle) \cong E \times \bP^1$.

Now $G=\Delta \times \langle g \rangle$ by Lemma \ref{lem_comm}.
So Lemma \ref{lem_Delta} implies that $G$ is isomorphic to one of the following:
$$\bZ/2,\ (\bZ/2)^2,\ (\bZ/2)^3,\ \bZ/3,\ (\bZ/3)^2,\ \bZ/4,\ \bZ/2 \times \bZ/4,\ \bZ/6.$$
Moreover, $N$ is a non-trivial cyclic subgroup of $G$.
So the third assertion follows.
\end{proof}

Set $K=\Ker(G/N \to \Aut(\bP^1))$.
Then $(E \times \bP^1)/K=E' \times \bP^1$ for another elliptic curve $E'$, 
and the action of $G'=(G/N)/K$ on $E' \times \bP^1$ is induced from 
a faithful action $G' \to \Aut(E')$ by translations and a faithful action $G' \to \Aut(\bP^1)$.
As a result, we obtain the following: 

\begin{prop}\label{prop_candidates}
Let $X$ be a smooth finite quotient of an abelian surface with $q^\circ(X)=1$.
Then there exists an elliptic curve $E$ and a finite group $G$ such that 
\begin{itemize}
\item
$G$ acts faithfully on $E$  by translations.
\item
$G$ acts faithfully on $\bP^1$.
\item
$G \cong \{1\},\ \bZ/2,\ \bZ/3,\ \bZ/4$, or $(\bZ/2)^2$.
\item
$X \cong (E \times \bP^1)/G$.
\end{itemize}
\end{prop}

Let us investigate the structures of $(E \times \bP^1)/(\bZ/n)$ and $(E \times \bP^1)/(\bZ/2)^2$.

\begin{lem}\label{lem_quotbyZ/n}
Let $E$ be an elliptic curve and $G \cong \bZ/n$ the $n$-th cyclic group such that 
\begin{itemize}
\item
$G$ acts faithfully on $E$  by translations.
\item
$G$ acts faithfully on $\bP^1$.
\end{itemize}
Then $(E \times \bP^1)/G \cong \bP_{E'}(\cO_{E'} \oplus L)$,
where $E'=E/G$ and $L$ is a torsion line bundle of order $n$ on $E'$.
\end{lem}

\begin{proof}
Set $G=\langle t_a \times \vph \rangle$, then $E'=E/\langle a \rangle$.
Write $E=\bC/\Lambda$ and let $v \in \frac{1}{n} \Lambda$ be a vector that induces $a \in E$.
Now $\lambda=nv \in \Lambda$ is a primitive vector, so we can take $\mu \in \Lambda$ such that 
$\Lambda=\langle \lambda, \mu \rangle$.
Then $E'=\bC/\Lambda'$ with $\Lambda'=\langle v, \mu \rangle$ and the quotient map $q: E \to E'$ is induced by the identity map $\id_{\bC}: \bC \to \bC$.

Take a character $\chi: \Lambda' \to \bC_1$ defined by 
$\chi(kv+l \mu)=\zeta_n^k$, where $\zeta_n$ is a primitive $n$-th root of unity.
Then the line bundle $L$ defined by $\chi$ is a torsion line bundle of order $n$ such that $q^* L=\cO_E$.
Now we have the following diagram of a fiber product:
$$
\xymatrix@C=40pt{
 E \times \bP^1 \ar@{}[rd]|{\Box} \ar[r]^{p} \ar[d]_{\pr_1} & \bP_{E'}(\cO_{E'} \oplus L) \ar[d]^{\pi}  \\
E \ar[r]^q & E'
}
$$
Here $q$ is a finite Galois cover with Galois group $G' \cong \bZ/n$,
so $p$ is also Galois with the same Galois group.
Lemma \ref{lem_split} implies that $G'$ acts on $\bP^1$.

Set $N=\Ker(G \to \Aut(\bP^1))$.
Then $q$ splits as $E \to E'' \to E'$ and the pull-back of $\cO_{E'}\oplus L$ to $E''$ is trivial.
But the order of $L$ is $n$, so $N=\{1\}$.
Therefore $G' \to \Aut(\bP^1)$ is also faithful.
The faithful action of $\bZ/n$ on $\bP^1$ is unique up to conjugacy (cf.~Lemma \ref{lem_unique}),
so $(E \times \bP^1)/G \cong (E \times \bP^1)/G'=\bP_{E'}(\cO_{E'} \oplus L)$.
\end{proof}

\begin{lem}\label{lem_quotby(Z/2)^2}
Let $E$ be an elliptic curve and $G$ be a group with $G \cong (\bZ/2)^2$ such that 
\begin{itemize}
\item
$G$ acts faithfully on $E$  by translations.
\item
$G$ acts faithfully on $\bP^1$.
\end{itemize}
Then $(E \times \bP^1)/G \cong \Sym^2 E$.
\end{lem}

\begin{proof}
The faithful action of $(\bZ/2)^2$ on $\bP^1$ is unique up to conjugacy (cf.~Lemma \ref{lem_unique}).
Moreover, $E[2] \cong (\bZ/2)^2$ and so the faithful action of $(\bZ/2)^2$ on $E$ by translations is essentially unique.
Hence the quotient  $(E \times \bP^1)/(\bZ/2)^2$ satisfying the above conditions is unique if it exists.
Therefore it is enough to show that $\Sym^2 E$ can be represented as 
the quotient  $(E \times \bP^1)/(\bZ/2)^2$ satisfying the conditions in the statement.

Let $\nu: E \times E \to E$ be a homomorphism defined by $\nu(x,y)=(x+y, x-y)$.
Set 
$$G=\{t_a \times (t_a \circ [\vep]) \in \Aut(E \times E) \mid a \in E[2],\ \vep=\pm 1 \}.$$
Let $\fS_2$ be the second symmetric group that acts on $E \times E$ by permutation.
It is easy to see that $\nu$ induces an isomorphism $\overline{\nu}: (E \times E)/G \to (E \times E)/\fS_2=\Sym^2 E$.

Set $N=\{ 1 \times [\vep] \mid \vep=\pm 1\}$, which is a normal subgroup of $G$.
Then $(E \times E)/N \cong E \times \bP^1$ and $G/N \cong (\bZ/2)^2$.
Clearly $G/N$ acts faithfully on both $E$ and $\bP^1$, so we obtain the desired representation of $\Sym^2 E$.
\end{proof}

By Proposition \ref{prop_candidates} and the above two lemmas,
we complete the ``only if'' part of Theorem \ref{thm_main}.
We already know that the surfaces listed in Theorem \ref{thm_main} are actually finite quotients of abelian surfaces,
except the third class in the list.
So the proof of Theorem \ref{thm_main} is completed once we prove the following: 

\begin{prop}\label{prop_realize}
Let $E$ be an elliptic curve and $L$ be a torsion line bundle of order at most $4$ on $E$.
Then $\bP_E(\cO_E \oplus L)$ is a finite quotient of an abelian surface.
\end{prop}

\begin{lem}\label{lem_conv}
Let $E$ be an elliptic curve and $L$ a torsion line bundle of order $n$.
Then there exists an elliptic curve $E'$ and a torsion point $a \in E'$ of order $n$ such that 
$E=E'/\langle a \rangle$ and the pull-back of $L$ to $E'$ is trivial.
\end{lem}

\begin{proof}
Set $E=\bC/\Lambda$.
Then $L$ is defined by a character $\chi: \Lambda \to \bC_1$ of order $n$.
Set $\Lambda=\langle \lambda_1, \lambda_2 \rangle$ and $\chi(\lambda_i)=\zeta_i$.
Now $\zeta_1, \zeta_2$ are $n$-th root of unity and one of them is primitive.
So we may assume that $\zeta_1$ is a primitive $n$-th root of unity.
Then $\zeta_2=\zeta_1^k$ for some $k \in \bZ$.
Set $\lambda_2'=\lambda_2-k\lambda_1$,
then $\Lambda=\langle \lambda_1, \lambda_2' \rangle$ and $\chi(\lambda_2')=1$.

Set $\Lambda'=\langle n\lambda_1, \lambda_2' \rangle \subset \Lambda$.
Then the identity map $\id_{\bC}: \bC \to \bC$ induces $q: E'=\bC/\Lambda' \to E$.
Let $a \in E'$ be a point defined by $\lambda_1 \in \bC$, then $a$ is a torsion point of order $n$ and $E=E'/\langle a \rangle$.
Moreover $q^*L=\cO_{E'}$ since $\chi|_{\Lambda'}=\id_{\Lambda'}$.
\end{proof}

\begin{lem}\label{lem_intermediate}
Let $E$ be an elliptic curve and $a \in E$ a torsion point of order $n \le 4$.
Then there exists an elliptic curve $F$, a finite subgroup $G < \Aut(E \times F)$ and a normal subgroup $N \lhd G$ 
such that 
\begin{itemize}
\item
$(E \times F)/N  \cong E \times \bP^1$.
\item
$G/N \cong \bZ/n$.
\item
$G/N$ acts faithfully on $E$  by translation by $a$.
\item
$G/N$ acts faithfully on $\bP^1$.
\end{itemize}
\end{lem}

\begin{proof}
If $n=1$, then $G=N=\{ 1 \times 1, 1 \times [-1]\}$ satisfies the claim.

Assume $n=2$ or $3$.
Take an elliptic curve $F$ with a group automorphism $\vph$ of order $2n$ and set 
$G=\langle t_a \times \vph \rangle$, $N=\langle 1 \times \vph^n \rangle$.
Then the required conditions are satisfied.

Assume $n=4$.
Take an elliptic curve $F$ with a group automorphism $\vph$ of order $4$, 
$a_0 \in E$ with $2a_0=a$, and a torsion point $b \in F$ of order $2$.
Set
$$G=\langle t_a \times t_b,\ t_{a_0} \times \vph \rangle,\ N=\Ker(G \to \Aut(E)).$$
Then $G = \langle t_a \times t_b \rangle \times \langle t_{a_0} \times \vph \rangle \cong \bZ/2 \times \bZ/4$ 
and $N=\langle 1 \times (t_b \circ \vph^2) \rangle$.
It is easy to see that the required conditions are satisfied.
\end{proof}

\begin{proof}[Proof of Proposition {\ref{prop_realize}}]
Lemma \ref{lem_conv} implies that 
there exists an elliptic curve $E'$ and a torsion point $a \in E'$ of order $n$ such that 
$E=E'/\langle a \rangle$ and the pull-back of $L$ to $E'$ is trivial.
So $\bP_E(\cO_E \oplus L)$ is the quotient of $E' \times \bP^1$ by a finite group $G \cong \bZ/n$ such that 
\begin{itemize}
\item
$G$ acts faithfully on $E'$  by the translation by $a$.
\item
$G$ acts faithfully on $\bP^1$.
\end{itemize}
On the other hand, Lemma \ref{lem_intermediate} implies that 
there exists an elliptic curve $F$, a finite group $G' < \Aut(E' \times F)$ and a normal subgroup $N \lhd G'$ 
such that 
\begin{itemize}
\item
$(E' \times F)/N  \cong E' \times \bP^1$.
\item
$G'/N \cong \bZ/n$.
\item
$G'/N$ acts faithfully on $E'$ by the translation by $a$.
\item
$G'/N$ acts faithfully on $\bP^1$.
\end{itemize}
Since the faithful action of $\bZ/n$ on $\bP^1$ is unique up to conjugacy (cf.~Lemma \ref{lem_unique}), 
the actions of $G$ and $G'/N$ on $E' \times \bP^1$ are conjugate to each other.
So we have 
$$(E' \times F)/G' \cong (E' \times \bP^1)/(G'/N) \cong (E' \times \bP^1)/G \cong \bP_E(\cO_E \oplus L).\qedhere$$
\end{proof}

Finally, we prove Theorem \ref{thm_orbstr}.

\begin{proof}[Proof of Theorem \ref{thm_orbstr}]
Use the quotient presentation $X\cong(E\times F)/G$ constructed above, 
and set
\[
N=\Ker(G\to\Aut(E))
\]
as in Lemma \ref{lem_reduction}.
By Lemma \ref{lem_reduction}, $N$ is a non-trivial cyclic normal subgroup of $G$ and $Y=(E\times F)/N$ is isomorphic to $E\times\bP^1$.
Thus (1) holds.

The induced action of $G/N$ on $E$ is faithful and is given by translations.
Hence every non-trivial element of $G/N$ acts on $E$ by a non-trivial translation,
and therefore has no fixed point on $E\times\bP^1$.
It follows that $\eta:Y\to X$ is finite \'etale.
This proves (2).

Let $1\times h$ be a generator of $N$.
By the proof of Lemma \ref{lem_reduction}, $h$ has non-trivial holonomy part.
Hence $h$ has a fixed point on $F$.
After conjugating the action on $F$ by a translation, we may assume that $h$ is a
group automorphism.
Therefore the order $n:=|N|=\ord(h)$ is one of $2,3,4,6$.

Let
\[
q:F\to F/N\cong\bP^1
\]
be the quotient map, let $p_1,\dots,p_s$ be its branch points, and let $e_i$ be the
ramification index over $p_i$.
Since $q$ is a Galois cover of degree $n$, there are $n/e_i$ points over $p_i$, each
with ramification index $e_i$.
Thus the Riemann--Hurwitz formula gives
\[
0=2g(F)-2
=-2n+\sum_{i=1}^s\frac{n}{e_i}(e_i-1),
\]
and hence
\[
\sum_{i=1}^s\left(1-\frac1{e_i}\right)=2.
\]
We now determine the $e_i$ using Lemma \ref{lem_table}.

If $n=2$, the four fixed points of $h$ give four branch points, all with ramification
index $2$. Thus the ramification indices are $(2,2,2,2)$.

If $n=3$, the three fixed points of $h$ give three branch points, all with ramification
index $3$. Thus the ramification indices are $(3,3,3)$.

Assume that $n=4$.
The two fixed points of $h$ give two branch points with ramification index $4$.
The involution $h^2$ has four fixed points. The remaining two points are exchanged by
$h$, so they form one $N$-orbit and give one branch point with ramification index $2$.
Thus the ramification indices are $(2,4,4)$.

Finally, suppose that $n=6$.
The unique fixed point of $h$ gives one branch point with ramification index $6$.
The automorphism $h^2$ has order $3$ and three fixed points. The two which are not
fixed by $h$ form one $N$-orbit and give one branch point with ramification index $3$.
Similarly, $h^3$ has order $2$ and four fixed points. The three which are not fixed by
$h$ form one $N$-orbit and give one branch point with ramification index $2$.
Thus the ramification indices are $(2,3,6)$.

Every ramification point of $q$ is fixed by a non-trivial element of $N$, hence by a
non-trivial power of $h$, so the above lists exhaust all ramification points.
This proves (3).

Let
\[
\Delta_Y:=\sum_{i=1}^s\left(1-\frac1{e_i}\right)D_i.
\]
Since $N$ acts trivially on $E$, the map $\rho$ is the product of the identity of $E$
and the quotient map $F\to F/N$.
Thus, under $Y\cong E\times\bP^1$, each branch divisor $D_i$ is
$E\times\{p_i\}$.
The Hurwitz formula gives
\[
K_{E\times F}=\rho^*(K_Y+\Delta_Y)
\]
and also
\[
K_{E\times F}
=\pi^*(K_X+\Delta_\pi)
=\rho^*\eta^*(K_X+\Delta_\pi).
\]
Since $\eta$ is \'etale, $K_Y=\eta^*K_X$.
Therefore $\rho^*\Delta_Y=\rho^*\eta^*\Delta_\pi$.
As $\rho$ is finite and surjective, pull-back is injective on $\bQ$-divisors, 
and hence $\eta^*\Delta_\pi=\Delta_Y$.
\end{proof}

\begin{data}
My manuscript has no associated data.
\end{data}

\begin{conflict}
The author declares that there is no conflict of interest.
\end{conflict}


\end{document}